\documentclass[a4paper,reqno]{amsart}

\usepackage{amsthm}
\usepackage{amsmath}
\usepackage{amssymb}
\usepackage{centernot}   
\usepackage{graphicx}        
\usepackage{multicol}        
\usepackage[all,ps]{xy}
\usepackage{layout}
\usepackage{booktabs}

\newcommand{\CC}{{\mathbb C}}

\newcommand{\cC}{{\mathcal C}}

\newcommand{\cE}{{\mathcal E}}
\newcommand{\cF}{{\mathcal F}}

\newcommand{\cI}{{\mathcal I}}

\newcommand{\cM}{{\mathcal M}}

\newcommand{\cO}{{\mathcal O}}

\newcommand{\cV}{{\mathcal V}} 
\newcommand{\cW}{{\mathcal W}}
\newcommand{\cX}{{\mathcal X}} 

\newcommand{\cZ}{{\mathcal Z}}

\newcommand{\dra}{\dashrightarrow}

\newcommand{\es}{\emptyset}
\newcommand{\Ext}{\text{Ext}}

\newcommand{\gF}{\mathfrak{F}}

\newcommand{\gM}{\mathfrak{M}}

\newcommand{\gz}{\mathfrak{z}}

\newcommand{\hra}{\hookrightarrow}
\newcommand{\la}{\langle}

\newcommand{\lra}{\longrightarrow}

\newcommand{\n}{\noindent}
\newcommand{\NN}{{\mathbb N}}
\newcommand{\ov}{\overline}
\newcommand{\PP}{{\mathbb P}}
\newcommand{\QQ}{{\mathbb Q}}
\newcommand{\ra}{\rangle}
\newcommand{\RR}{{\mathbb R}}

\newcommand{\wh}{\widehat}
\newcommand{\wt}{\widetilde}
\newcommand{\ZZ}{{\mathbb Z}}
\newcommand{\Gr}{\mathrm{Gr}}

\theoremstyle{plain}

\newtheorem{thm}{Theorem}[section]

\newtheorem{clm}[thm]{Claim}

\newtheorem{cnj}[thm]{Conjecture}
\newtheorem{crl}[thm]{Corollary}
\newtheorem{hyp}[thm]{Hypothesis}
\newtheorem{lmm}[thm]{Lemma}
\newtheorem{prp}[thm]{Proposition}
\newtheorem{prp-dfn}[thm]{Proposition-Definition}

\theoremstyle{definition}

\newtheorem{dfn}[thm]{Definition}

\theoremstyle{remark}

\newtheorem{expl}[thm]{Example}
\newtheorem{qst}[thm]{Question}
\newtheorem{rmk}[thm]{Remark}

\DeclareMathOperator{\Amp}{Amp}

\DeclareMathOperator{\ch}{ch}

\DeclareMathOperator{\gr}{gr}

\DeclareMathOperator{\Hom}{Hom}

\DeclareMathOperator{\im}{Im}

\DeclareMathOperator{\Pic}{Pic}

\DeclareMathOperator{\rk}{rk}

\DeclareMathOperator{\td}{Td}

\usepackage{ifthen}
\usepackage{rotating}
\usepackage{supertabular}
\newcommand{\cit}[1]{{\rm \textbf{#1}}}
\newcommand{\Ref}[2]{\cit{%
\ifthenelse{\equal{#1}{thm}}{Theorem}{}%
\ifthenelse{\equal{#1}{ass}}{Assumption}{}%
\ifthenelse{\equal{#1}{chp}}{Chapter}{}%
\ifthenelse{\equal{#1}{prp}}{Proposition}{}%
\ifthenelse{\equal{#1}{lmm}}{Lemma}{}%
\ifthenelse{\equal{#1}{cnj}}{Conjecture}{}%
\ifthenelse{\equal{#1}{crl}}{Corollary}{}%
\ifthenelse{\equal{#1}{dfn}}{Definition}{}%
\ifthenelse{\equal{#1}{expl}}{Example}{}%
\ifthenelse{\equal{#1}{hyp}}{Hypothesis}{}%
\ifthenelse{\equal{#1}{rmk}}{Remark}{}%
\ifthenelse{\equal{#1}{clm}}{Claim}{}%
\ifthenelse{\equal{#1}{exe}}{Exercise}{}%
\ifthenelse{\equal{#1}{qst}}{Question}{}%
\ifthenelse{\equal{#1}{sec}}{Section}{}%
\ifthenelse{\equal{#1}{subsec}}{Subsection}{}%
\ifthenelse{\equal{#1}{univ}}{Universal Property}{}%
\ifthenelse{\equal{#1}{trm}}{Terminology}{}%
\ifthenelse{\equal{#1}{tbl}}{Table}{}%
\  \ref{#1:#2}%
}}

\usepackage{multirow}

 \makeindex
 
\begin{document}
 \title{Moduli of sheaves  and the Chow group of $K3$ surfaces}
 \author{Kieran G. O'Grady\\\\
\lq\lq Sapienza\rq\rq Universit\`a di Roma}
\date{May 22 2012}
\thanks{Supported by
 PRIN 2010}
  \maketitle
 \tableofcontents
 \section{Introduction}\label{sec:prologo}
 \setcounter{equation}{0}
Let $X$ be a projective complex $K3$ surface. Let $CH_n(X)$ be the Chow group of dimension-$n$ cycles on $X$ modulo rational equivalence. Beauville and Voisin~\cite{beauvoisin} singled out a class $c_X\in CH_0(X)$ of degree $1$: it is represented by any point lying on an arbitrary  rational  curve (an irreducible curve whose normalization is rational). The class $c_X$ has the following remarkable property.  
\begin{equation}\label{interseco}
\text{Let $D_1,D_2\in CH_1(X)$: then $D_1\cdot D_2\in\ZZ c_X$.}
\end{equation}
Moreover $c_2(X)=24 c_X$. (Conjecturally the Chow ring of Hyperk\"ahler varieties has similar properties, see~\cite{beauville,voisin2}.)  In particular one has the {\it Beauville-Voisin ring} $CH^0(X)\oplus CH^1(X)\oplus \ZZ c_X$. 
 Huybrechts~\cite{huybrechts} proved that if $E$ is a spherical object in the bounded derived category of $X$ then the Chern character of $E$ belongs to the Beauville-Voisin ring provided $\Pic(X)$ has rank at least $2$ or $c_1(E)\equiv \pm 1\pmod{\rk(E)}$ in case $\Pic(X)\cong\ZZ$. A rigid simple vector-bundle on $X$ is a particular case of spherical object. One may summarize Huybrechts' main result  
 as follows: if  $F_1$, $F_2$ are rigid vector-bundles on $X$  (the additional hypotheses mentioned above are  in force)    then $c_2(F_1)=c_2(F_1)+a c_X$ where $a:=(\deg c_2(F_1)-\deg c_2(F_1))$. We believe that the following more general statement (with no additional hypotheses) holds. Let $\gM^{\rm st}_1$ and $\gM^{\rm st}_2$ be moduli spaces of stable pure sheaves on $X$ (with fixed cohomological Chern characters)  and suppose that $\dim\gM^{st}_1=\dim\gM^{st}_2$:  then the   subset of $CH_0(X)$ whose elements are $c_2(F_1)$ where $[F_1]\in\ov{\cM}^{\rm st}_1$ (the closure of $\gM^{\rm st}_1$ in the moduli space of semistable sheaves) is  equal to the  subset of $CH_0(X)$ whose elements are   $c_2(F_2)+a c_X$ where $[F_2]\in\ov{\cM}^{\rm st}_2$ and $a:=(\deg c_2(F_1)-\deg c_2(F_2))$ (notice that $a$ is independent of $F_1$ and $F_2$). We will prove that the above statement holds under some additional assumptions.
Before formulating  our main result 
we will define certain subsets of $CH_0(X)$. 
\begin{dfn}
Let $S_g(X)\subset CH_0(X)$ be the set of classes $[Z]+a c_X$ where $Z=p_1+\ldots+p_g$ is an  effective $0$-cycle of degree $g$ and $a\in\ZZ$. 
\end{dfn}
Notice that $S_0(X)=\ZZ c_X$. 
\begin{clm}\label{clm:perchegi}
Let $C$ be an irreducible smooth projective curve of genus $g$ and $f\colon C\to X$ be a non-constant map. Then $f_{*}CH_0(C)\subset S_g(X)$.
\end{clm}
\begin{proof}
There exists $p\in C$ such that $f_{*}[p]=c_X$. In fact let $H$ be a primitive ample divisor on $X$, by~\cite{mori-mukai} there exists $D\in  |H|$ whose irreducible components are rational curves. Since $f$ is not constant and $D$ is ample $f(C)\cap D\not=\es$: if $p\in f^{-1}( D)$ then  $i_{*}[p]=c_X$. Now let $\gz\in CH_0(C)$. By Riemann-Roch there exists an effective cycle $p_1+\ldots+ p_g$ on $C$ such that $\gz=[p_1+\ldots+p_g]+(\deg\gz -g) p$: thus $f_{*}\gz=([f(p_1)+\ldots+ f(p_g)]+(\deg\gz -g) c_X)\in S_g(X)$. 
\end{proof}
Multiplication by $\ZZ$ maps $S_g(X)$ to itself - see~\Ref{crl}{anakin}. Thus we may say that  $S_g(X)$  is a cone; on the other hand $S_g(X)$ is a subgroup of $CH_0(X)$ only if $g=0$.  We have a  filtration
\begin{equation}\label{filtrochow}
S_0(X)\subset S_1(X)\subset \ldots \subset S_g(X)\subset S_{g+1}(X)\subset\ldots \subset CH_0(X).
\end{equation}
In fact let $\gz=([p_1+\ldots+p_g]+a c_X)\in S_g(X)$. Let $p_{g+1}\in X$ be a point lying on a rational curve: then $[p_{g+1}]=c_X$ and hence $\gz=([p_1+\ldots+p_g+p_{g+1}]+(a-1) c_X)\in S_{g+1}(X)$. This proves~\eqref{filtrochow}. We also have that
\begin{equation}\label{esaurisce}
\bigcup_{g=0}^{\infty} S_g(X)=CH_0(X).
\end{equation}
In fact let $\gz\in CH_0(X)$. There exist a smooth curve $\iota\colon C_0\hra X$ of genus $g$ and a cycle $D_0\in Z^1(C_0)$ such $\gz=[\iota_{*}D_0]$. By~\Ref{clm}{perchegi} we get that $\gz\in S_g(X)$; this proves~\eqref{esaurisce}. Next we recall 
that the {\it Mukai pairing} on $H^{\bullet}(X;\ZZ)$ is the symmetric bilinear form defined by
\begin{equation}\label{mukprod}
\la \alpha,\beta\ra :=-\int_X \alpha^{\vee}\cup \beta,\qquad
 (\alpha_0+\alpha_2+\alpha_4)^{\vee}:= \alpha_0-\alpha_2+\alpha_4,\quad \alpha_p\in H^p(X;\ZZ).
\end{equation}
Let
\begin{equation}\label{vettomuk}
v=(r,\ell,s)\in H^{\bullet}(X;\ZZ).
\end{equation}
(We identify $H^4(X;\ZZ)$ with $\ZZ$ via the orientation class.) 
\begin{dfn}\label{dfn:vettmuk}
A \emph{Mukai vector} (for $X$) is a $v$ as in~\eqref{vettomuk} such that the following hold:
  \begin{enumerate}
\item[(1)]
$r\ge 0$, 
\item[(2)]
$\ell\in H^{1,1}_{\ZZ}(X)$,
\item[(3)]
if $r=0$ then $\ell$ is effective.   
\end{enumerate}
\end{dfn}
 Given a coherent sheaf $F$ on $X$ the {\it Mukai vector of $F$} is
\begin{equation}\label{vettoremuk}
v(F):=(\ch^{\rm hom}_0(F)+\ch^{\rm hom}_1(F)+\ch^{\rm hom}_2(F))\cup\sqrt{\td_X}
\end{equation}
where $c_p^{\rm hom}(F)\in H^{2p}(V;\ZZ)$ is  the  topological $p$-th Chern class of $F$. 
  Suppose that $v\in H^{\bullet}(X;\ZZ) $ is a Mukai vector and $H$ is an ample divisor on $X$. Let $\gM_v(X,H)$ be the   moduli space of $S$-equivalence classes of pure $H$-semistable sheaves on $X$ with $v(F)=v$, see~\cite{huy-lehn,simp}. Thus  $\gM_v(X,H)$  is a projective complex scheme. Let $\gM_v(X,H)^{\rm st}$ be the open subscheme of $\gM_v(X,H)$ 
  parametrizing isomorphism classes of pure $H$-stable sheaves. Suppose that  $\gM_v(X,H)^{\rm st}$ is  not empty: then  it is  smooth of   pure dimension given by 
\begin{equation}
\dim\gM_v(X,H)^{\rm st}=2+v^2=2d(v). 
\end{equation}
(We let $v^2:=\la v,v\ra$.)
Notice that   $d(v)$ is an integer because the Mukai pairing is even. 
We let $\ov{\gM}_v(X,H)^{\rm st}$ be the closure of $\gM_v(X,H)^{\rm st}$ in $\gM_v(X,H)$. 
Let
\begin{equation}
c_2(v):=  r+\frac{\ell \cdot\ell}{2}-s. 
\end{equation}
Thus $c_2(v)$ is the degree of $c_2(F)$ where  $F$ is a coherent sheaf such that $v(F)=v$. 
\begin{cnj}\label{cnj:solodim}
Let $X$ be a a projective complex $K3$ surface and $H$ an ample divisor on $X$. Let $v\in H^{\bullet}(X;\ZZ)$ be a Mukai vector. Suppose that  $\gM_v(X,H)^{\rm st}$ is  not empty. Then
\begin{equation}\label{anvedi}
\{c_2(F) \mid [F]\in\ov{\gM}_v(X,H)^{\rm st} \}=\{\gz\in S_{d(v)}(X) \mid \deg\gz=c_2(v)\}.
\end{equation}
(Here $\deg\colon CH_0(X)\to\ZZ$ is the degree homomorphism.)
\end{cnj}
\begin{rmk}
Let $[F]\in\gM_v(X,H)$ with $F$ not $H$-stable i.e.~properly $H$-semistable. The same point of $\gM_v(X,H)$ is represented by any $H$-semistable pure sheaf $G$ which is $S$-equivalent to $F$ i.e.~such that  $\gr^{JH}(F)\cong \gr^{JH}(G)$ where $\gr^{JH}(F)$, $\gr^{JH}(G)$ are the  the direct-sums of the successive quotients of  Jordan-Holder  filtrations of $F$ and $G$. It follows that although $F$, $G$ may not be isomorphic  the Chern classes $c_2(F)$ and $c_2(G)$ are equal. This shows that the left-hand side of~\eqref{anvedi} is well-defined. 
\end{rmk}
The following is the main result of the present paper. 
\begin{thm}\label{thm:valese}
Let $X$ be  a projective complex $K3$ surface and  $v=(r,\ell,s)$ be  a Mukai vector. Let $H$ be an ample  divisor on $X$. Suppose that $\gM_v(X,H)^{\rm st}$ is not empty and that one of the following holds:
\begin{enumerate}
\item[(1)]
$\ell=c_1^{\rm hom}(\cO_X(H))$, $\ell$ is primitive  and  $s\ge 0$. 
\item[(2)]
The Picard number of $X$ is at least $2$, $r$ is  coprime to the divisibility of $\ell$ and $H$ is $v$-generic 
(see~\Ref{subsec}{invarianza} for the relevant definition). 
\item[(3)]
$r\le 2$ and moreover $H$ is $v$-generic if $r=2$.
\end{enumerate}
Then~\eqref{anvedi} holds. 
\end{thm}
A few comments on~\Ref{thm}{valese}. Suppose that   $-2\le v^2$ and $H$ is $v$-generic: if $r>0$ is coprime to the divisibility of $\ell$ or  if $r=2$ and $v\not=(2,2\ell_0,\ell_0\cdot\ell_0)$ then  $\gM_v(X,H)^{\rm st}$ is not empty  - see~\Ref{thm}{modce} and~\Ref{prp}{kodzero}. If $r=0$  the proof that~\eqref{anvedi} holds is an easy exercise, if $r=1$ (stability is not an issue in this case)  then~\eqref{anvedi} holds by definition. Now assume that  $r\ge 2$: the starting idea in the proof  is as follows. Let $v=(r,\ell,s)$ be a Mukai vector such that the following hold: $\ell=c_1^{\rm hom}(L)$ where $L$ is ample and $s\ge 0$. Let $[F]\in \gM_v(X,H)^{\rm st}$. Then   $H^2(F)$ vanishes by stability and hence Hirzebruch-Riemann-Roch gives that $h^0(F)\ge \chi(F)=r+s$.  Applying  Hirzebruch-Riemann-Roch and Kodaira vanishing to compute $h^0(L)$ we get that
\begin{equation}\label{disegchiave}
\dim\Gr(r,H^0(F))\ge rs=\dim|L|-d(v). 
\end{equation}
Now  assume that for every  $U\in \Gr(r,H^0(F))$ the tautological map $\varphi^U_F\colon U\otimes\cO_X\to F$ is generically an isomorphism: then we have a regular map
\begin{equation*}
\begin{array}{ccl}
\Gr(r,H^0(F)) & \overset{\lambda_F}{\lra} & |L| \\
U & \mapsto & V(\det\varphi_F^U).
\end{array}
\end{equation*}
The pull-back by $\lambda_F$ of the hyperplane class on $|L|$ is linearly equivalent to the Pl\"ucker hyperplane class on $\Gr(r,H^0(F))$: it follows that $\dim(\im\lambda_F)=\dim\Gr(r,H^0(F))\ge rs$. On the other hand  there exists a closed subset $\Sigma_{d(v)}(X,L)\subset |L|$ of dimension at least $d(v)$ with the property that for every $C\in \Sigma_{d(v)}(X,L)$ the push-forward $CH_0(C)\to CH_0(X)$ has image contained in $S_{d(v)}(X)$ - this follows from~\Ref{clm}{perchegi} and known results on Severi varieties in complete linear systems on $K3$'s, see~\Ref{prp}{anakin} and~\eqref{almenogizero}. 
 By~\eqref{disegchiave} it follows  there exists $C_0\in (\im\lambda_F)\cap\Sigma_{d(v)}(X,L)$. Let $U_0\in \Gr(r,H^0(F))$ be such that $C_0=V(\det\varphi_F^{U_0})$. Applying Whitney's formula to the exact sequence $0\to U_0\otimes\cO_X\to F\to \xi\to 0$ we get that  $c_2(F)\in S_{d(v)}(X)$ if $c_2(\xi)\in S_{d(v)}(X)$: the latter  holds because of the stated property of curves (such as $C_0$) that belong to $\Sigma_{d(v)}(X,L)$. The proof sketched above - together with some extra work -  gives Items(1) and~(2) of~\Ref{thm}{valese}. 
In general (say for $\ell$ highly divisible) we will only get a rational map $\lambda_F\colon\Gr(r,H^0(F)) \dra |L| $. It might be quite difficult to resolve the indeterminacies of that map in order to determine the dimension of the image. We will show how to circumvent that problem when the rank is $2$ - that gives Item~(3)  of~\Ref{thm}{valese}.  

\vskip 2mm
\n
{\bf Notation and conventions.}
Schemes are over $\CC$. Points are closed (geometric) points unless we specify differently. By a sheaf on a scheme we always mean a {\bf coherent} sheaf. For a smooth projective variety $X$ we let $\rho(X)$ be its Picard number i.e.~the rank of the Neron-Severi group: thus $\rho(X)$ is equal to $h^{1,1}_{\QQ}(X):=\dim_{\QQ}(H^{1,1}(X)\cap H^2(X;\QQ))$.
By a $K3$ surface we always mean a (complex) projective $K3$ surface. 
Let $L$  be a line-bundle on a  $K3$ surface $X$: we let
\begin{equation}
g(L):=\chi(L)-1=\frac{1}{2}\deg(L\cdot L)+1.
\end{equation}
If $D$ is a divisor on $X$ we let $g(D):=g(\cO_{X}(D))$. Let   $C\subset X$ be an integral curve: then $g(C)$ is the arithmetic genus of $C$. 

\vskip 2mm
\n
{\bf Acknowledgment:} It is a pleasure to thank Daniel Huybrechts for stimulating conversations. Thanks to Claire Voisin for mentioning the filtration of~\Ref{sec}{fantasia}.
\section{Preliminaries}\label{sec:sicomincia}
\setcounter{equation}{0}
\subsection{Generalities on the filtration}\label{subsec:fattigen}
\setcounter{equation}{0}
  Let $X$ be a $K3$ surface. Let $C$ be an effective divisor on $X$. We view $C$ as a (purely) $1$-dimensional subscheme of $X$: let $\iota\colon C\hra X$ be the inclusion map. 
  \begin{clm}\label{clm:spingo}
Keep notation and assumptions  as above. Let $\xi$ be a sheaf on $C$. There exist $\eta\in CH_0(C)$ and $a\in \ZZ$ such that
\begin{equation}\label{clm:riemroch}
c_2(\iota_{*}\xi)=\iota_{*}\eta+a c_X.
\end{equation}
\end{clm}
  \begin{proof}
By~\eqref{interseco} we are free to tensor $\xi$ by an arbitrary invertible sheaf on $X$. Thus we may assume that $\xi$ is globally generated and hence there exists an exact sequence
\begin{equation*}
0\lra \oplus_{i=1}^n \cO_{C_i}^{r_i}\lra \xi \lra \zeta \to 0
\end{equation*}
 where  $C_i$ for $i=1,\ldots,n$ are the irreducible components of $C$ ($r_i$ is the rank of the restriction of $\xi$ to $C_i$) and
 $ \zeta$ has $0$-dimensional support. It follows that it suffices to check that $c_2(\iota_{*}\cO_{C_i})\in\ZZ c_X$: that follows at once from~\eqref{interseco}.
\end{proof}
 \begin{clm}\label{clm:specializzo}
Let  $f\colon \cX\to T$ be a  projective family of $K3$ surfaces i.e.~$f$ is projective, flat and the fibers are $K3$ surfaces. Let $\cZ\in CH^2(\cX)$. Suppose that there is a dense open subset $U\subset T$ such that  $Z_t:=\cZ|X_t\in CH^2(X_t)$ belongs to $S_g(X_t)$ for every $t\in U$. Then $Z_t\in S_g(X_t)$ for all $t\in T$.
\end{clm}
 \begin{proof}
The claim follows from the fact that the set of degree-$d$ effective $0$-cycles on a  variety $V$ belonging to a fixed linear equivalence class is a countable union of closed subsets of the symmetric product $V^{(d)}$. We give a proof for the reader's convenience. 
We may assume that $T$ is an irreducible curve. In particular $\deg Z_t$ is independent of $t\in T$: let $d:=\deg Z_t$. 
We are free to perform a base-change: thus we may assume that there exists a section $\pi\colon T\to\cX$ of $f$ such that $\pi(t)$ represents $c_{X_t}$ for every $t\in T$:  let $p_t:=\pi(t)$. 
Let $n,e\in\ZZ$ such that $n+d\ge 0$ and $e\ge 0$. We let $T[n,e]^0\subset T$ be the set of $t$ such that there exist  $W\in X_t^{(n)}$,  $Y\in X_t^{(g)}$ and a map $f\colon \PP^1\to X_t^{(n+d)}$ such that 
\begin{enumerate}
\item[(1)]
$(Z_t+W)$ and $(Y+(d-g) p_t+W)$ are effective, 
\item[(2)]
$f(0)=(Z_t+W)$ and $f(\infty)=(Y+(d-g) p_t+W)$.
\end{enumerate}
If $t\in T[n,e]^0$ then the class $[Z_t]$ belongs to $S_g(X_t)$. Let  $T[n,e]$ be the closure of $T[n,e]^0$ in $T$. By considering the relative Hilbert scheme parametrizing subschemes of $X_t^{(d+n)}$ for $t\in T$  with Hilbert polynomial $p(m):=em+1$ (or the relative  parameter space for  genus $0$ stable maps to $X_t^{(d+n)}$) we get that if   $t\in T[n,e]$ then $[Z_t]$ belongs to $S_g(X_t)$. Conversely - see  Example~1.6.3 of~\cite{fulton} - if $[Z_t]$ belongs to $S_g(X_t)$ then $t\in T[n,e]$ for some $n$ and $e$ as above.
 Thus 
\begin{equation*}
U\subset\bigcup_{n+d\ge 0\ 
e\ge 0}  T[n,e].
\end{equation*}
Since each $T[n,e]$ is closed and $U$ is open it follows that there exist $n_0,e_0$ such that $U\subset T[n_0,e_0]$. By hypothesis $U$ is dense in $T$ and hence $T=T[n_0,e_0]$: it follows that $Z_t\in S_g(X_t)$ for all $t\in T$.
\end{proof}
\begin{prp}\label{prp:bastacont}
Let $X$ be a $K3$ surface and  $v$ be  a Mukai vector for $X$. Let $H$ be an ample divisor on $X$. Suppose that $\gM_v(X,H)^{\rm st}$ is not empty. If
\begin{equation}\label{ciduecont}
\{c_2(F) \mid [F]\in\gM_v(X,H)^{\rm st} \}\subset\{\gz\in S_{d(v)}(X) \mid \deg\gz=c_2(v)\}
\end{equation}
then~\eqref{anvedi} holds.
\end{prp}
\begin{proof}
Let
\begin{equation}\label{leuzzi}
{\bf W}:=\{c_2(F) \mid [F]\in\ov{\gM}_v(X,H)^{\rm st} \}. 
\end{equation}
By hypothesis ${\bf W}\subset S_{d(v)}(X)$. Let  
\begin{equation*}
\begin{matrix}
 X^{[d(v)]} & \overset{\pi}{\lra} & S_{d(v)}(X) \\
 Z & \mapsto & [Z]+(c_2(v)-d(v))c_X
\end{matrix}
\end{equation*}
(Here $[Z]$ is the cycle-class associated to the scheme $Z$ i.e.~the class of $\sum_{p\in X}\ell(\cO_{Z,P})p$.) Arguing as in the proof of~\Ref{clm}{specializzo} we get that $\pi^{-1}{\bf W}$ 
is a countable union of closed subsets of $ X^{[d(v)]}$. It follows that there exists a closed 
 $V\subset \pi^{-1}{\bf W}$ such that $\pi(V)={\bf W}$. It suffices to prove that $\dim V=\dim X^{[d(v)]}=2d(v)$. Let $\sigma\in H^0(K_X)$ be a symplectic form. Then $\sigma$ induces a symplectic form $\sigma_{d(v)}$ on $X^{[d(v)]}$ - the {\it trace} of $\sigma$, see~\cite{mum} - and the 
{\it Mukai-Tyurin}   symplectic form   $\sigma_{v}$ on $\gM_v(X,H)^{\rm st}$, see~\cite{muk,tyur}. The two symplectic forms are compatible, up to a factor. This means that there exist a  smooth quasi-projective $\wt{\gM}_v$, a generically finite surjective map $\wt{\gM}_v\to\gM^{\rm st}_v$  and a map $q\colon \wt{\gM}_v\to X^{[d(v)]}$ such that 
\begin{equation}\label{paragono}
p^{*}\sigma_v=-4\pi^2 q^{*}\sigma_{d(v)}.
\end{equation}
In fact the above equation follows from Equation~(2-9) of~\cite{ograss} and arguments similar to those given in the proof of~\Ref{clm}{specializzo}. 
Now assume that $\dim V<\dim X^{[d(v)]}$. Then~\eqref{paragono} gives that $\sigma_v$ is everywhere degenerate - recall  that $\dim\gM_v(X,H)^{\rm st}=2d(v)=\dim X^{[d(v)]} $. That is a contradiction.
\end{proof}
\subsection{Severi varieties}\label{subsec:curvenodate}
\setcounter{equation}{0}
 Let $L$ be an ample line-bundle on $X$ and 
  $0\le\delta\le g(L)$. We let $V_{\delta}(X,L) \subset |L|$ be the Severi variety parametrizing
 integral curves whose geometric genus is $(g(L)-\delta)$ - thus $V_{\delta}(X,L)$ is locally closed. If $V_{\delta}(X,L)$ is non-empty then it has pure dimension 
\begin{equation*}
\dim V_{\delta}(X,L)=\dim|L|-\delta=g(L)-\delta.
\end{equation*}
By results of X.~Chen and Bogomolov - Hassett - Tscinkel we know that $V_{\delta}(X,L)$ is non-empty if $(X,L)$ is generic. Let us be more precise. 
Let 
\begin{equation}\label{mappapi}
\pi\colon\cX\lra T_{g}
\end{equation}
 be a complete  family of $K3$ surfaces with a polarization of degree $(2g-2)$ i.e.~the following hold:
\begin{enumerate}
\item[(1)]
$\pi$ is a projective and smooth map, we let $\cM$ be \lq\lq the\rq\rq relativley ample line-bundle.
\item[(2)]
Let $t\in T_g$: then  $X_t=\pi^{-1}(t)$ is a $K3$ surface. 
\item[(3)]
Let $t\in T_g$ and  $M_t:=\cM|_{X_t}$: then $c^{\rm hom}_1(M_t)$ is indivisible and $g(M_t)=g$.
\item[(4)]
 if $X$ is a $K3$ surface equipped with an indivisible ample line-bundle $M$ with $g(M)=g$  there exist $t\in T_{g}$ and an isomorphism $f\colon X\overset{\sim}{\lra} X_t$ such that $f^{*}M_t\cong M$.
\end{enumerate}
Such a family exists, moreover  we may assume that $T_{g}$  is  irreducible by the Global Torelli Theorem for $K3$ surfaces.  
Below is the result that we mentioned (see Ch.~11 of~\cite{huyonk3} for a detailed treatment  of the proof by Bogomolov - Hassett - Tschinkel).
\begin{thm}[Chen~\cite{chen},  Bogomolov - Hassett - Tschinkel~\cite{bht}]\label{thm:severiesiste}
Keep notation as above. Let $n>0$ be an integer. There exists an open dense $U_{g}(n)\subset T_{g}$ such that the following holds. Let $0\le\delta\le g(M^{\otimes n}_t)$ and   $t\in U_{g}(n)$: then $V_{\delta}(X_t, M^{\otimes n}_t)$ is non-empty.
\end{thm}
 Fix $0\le\delta\le g(M^{\otimes n}_t)$: the Severi varieties $V_{\delta}(X_t,M^{\otimes n}_t)$ for $t\in U_{g}(n)$ fit together to give 
 \begin{equation}
\cV_{\delta}(n)\lra U_{g}(n).
\end{equation}
  We let $\cW_{\delta}(n)$  be the closure of $\cV_{\delta}(n)$ in the projective bundle over $T_{g}$ with fiber $|M^{\otimes n}_t|$ over $t$. Thus we have a proper surjective map
   \begin{equation}
\rho_{\delta,n}\colon \cW_{\delta}(n)\lra T_{g}.
\end{equation}
 Let  $0\le g_0\le g(M^{\otimes n}_t)$ and $\delta_0:=g(M^{\otimes n}_t)-g_0$. 
Given $t\in T_{g}$ we let 
\begin{equation*}
 \Sigma_{g_0}(X_t,M^{\otimes n}_t):=\rho_{\delta_0,n}^{-1}(t).
\end{equation*}
\begin{rmk}\label{rmk:chiusurafibra}
If  $t\in T_g$ is generic then $\Sigma_{g_0}(X_t,M^{\otimes n}_t)$ is the closure of $V_{\delta_0}(X_t,M^{\otimes n}_t)$ in $| M^{\otimes n}_t |$.
\end{rmk}
By~\Ref{thm}{severiesiste} and a standard argument we get that 
\begin{equation*}
\dim\Sigma_{g_0}(X_t,M^{\otimes n}_t)=g_0,\quad \text{ $t\in T_g$  generic.}
\end{equation*}
(Pure dimension.) Since $\rho_{\delta_0,n}$ is a proper map we get that 
\begin{equation}\label{almenogizero}
\dim\Sigma_{g_0}(X_t,M^{\otimes n}_t)\ge g_0,\quad \forall t\in T_g.
\end{equation}
\begin{prp}\label{prp:anakin}
 Let $X$ be a  $K3$ surface. Let $L$ be an ample line-bundle on $X$. Let 
 $0\le g_0\le g( L )$. 
 Let $C\in \Sigma_{g_0}(X,L)$ and $\iota\colon C\hra X$ be the inclusion.  Then $\iota_{*} CH_0(C)\subset S_{g_0}(X)$. 
\end{prp}
\begin{proof}
There exists an ample line-bundle $M$ with $c_1^{\rm hom}(M)$ indivisible such that $L\cong M^{\otimes n}$. Let $g:=g(M)$. There exist $t\in T_g$ 
and an isomorphism $f\colon X\overset{\sim}{\lra} X_t$ such that $f^{*}M_t\cong M$.
Thus it suffices to prove~\Ref{prp}{anakin} for $X=X_t$ and $L=M_t^{\otimes n}$ where  $t\in T_g$. By~\Ref{clm}{specializzo} we may assume that $t$ is generic in $T_g$ and hence $\Sigma_{g_0}(X_t,M^{\otimes n}_t)$ is the closure of $V_{\delta_0}(X_t,M^{\otimes n}_t)$ in $| M^{\otimes n}_t |$ - 
see~\Ref{rmk}{chiusurafibra}. Again by~\Ref{clm}{specializzo} we may assume that $C\in V_{\delta_0}(X_t,M^{\otimes n}_t)$: in that case the result holds by~\Ref{clm}{perchegi}. 
\end{proof}
\begin{crl}\label{crl:anakin}
 Let $X$ be a  $K3$ surface. Multiplication by $\ZZ$ maps $S_{g_0}(X)$ to itself.
\end{crl}
\begin{proof}
It suffices to prove that $a[Z]\in S_{g_0}(X)$ for every $a\in\ZZ$ and $Z\in X^{(g_0)}$. 
Let $M$ be an ample line-bundle on $X$  with $c_1^{\rm hom}(M)$ indivisible. There 
exist $t\in T_g$ 
and an isomorphism $f\colon X\overset{\sim}{\lra} X_t$ such that $f^{*}M_t\cong M$.
Let $n>0$ such that $g(M_t^{\otimes n})\ge g_0$: then $\Sigma_{g_0}(X_t,M^{\otimes n}_t)$ is not empty and $\dim\Sigma_{g_0}(X_t,M^{\otimes n}_t)\ge g_0$ by~\eqref{almenogizero}. The set of $C\in \Sigma_{g_0}(X_t,M^{\otimes n}_t)$ containing a fixed  point of $X$ is a hyperplane section of $\Sigma_{g_0}(X_t,M^{\otimes n}_t)$: it follows that there exists $C\in \Sigma_{g_0}(X_t,M^{\otimes n}_t)$ containing the support of $Z$. Thus we may view $Z$ as a $0$-cycle on $C$. 
Let $\iota\colon C\hra X$ be the inclusion: then $a[Z]=\iota_{*}([aZ])$. By~\Ref{prp}{anakin} we get that $\iota_{*}([aZ])\in S_{g_0}(X)$.
\end{proof}
\subsection{The trivial cases}\label{subsec:casibanali}
\setcounter{equation}{0}
We will show that~\eqref{anvedi} holds if $r=0,1$. 
\begin{prp}
Let $X$ be a  $K3$ surface and  $v=(r,\ell,s)$ be  a Mukai vector for $X$ with $r\le 1$. Suppose that $\gM_v(X,H)^{\rm st}$ is not empty. Then~\eqref{anvedi} holds.
\end{prp}
\begin{proof}
Let $L$ be a line-bundle (unique up to isomorphism) such that $c_1^{\rm hom}(L)=\ell$. 
Suppose that $r=0$. Then
\begin{equation*}
2d(v)=\dim\ov{\gM}_v(X,H)^{\rm st}=2+\la v,v\ra=2+\ell\cdot\ell=2g(L).
\end{equation*}
Thus $d(v)=g(L)$. Suppose first that $\ell=0$. Then $L\cong\cO_X$. Since $g(\cO_X)=1$ we get that $d(v)=1$.  We have an isomorphism
\begin{equation*}
\begin{matrix}
S & \overset{\sim}{\lra} & \gM_v(S,H) \\
p & \mapsto & \CC_p
\end{matrix}
\end{equation*}
Since $c_2(\CC_p)$ is represented by $-p$ we get that $ c_2(\CC_p)\in S_1(X)$ by~\Ref{crl}{anakin}. By~\Ref{prp}{bastacont} we get that~\eqref{anvedi} holds. Now suppose that $\ell\not=0$ and hence $L\not\cong\cO_X$.
Let $[F]\in\ov{\gM}_v(X,H)^{\rm st}$: then there exist  $C\in |L|$ and a sheaf $\xi$ on $C$ such that $F=\iota_{*}\xi$ where $\iota\colon C\hra X$ is the inclusion map.  By~\Ref{clm}{riemroch} and~\Ref{prp}{anakin} we get that  
$c_2(\iota_{*} \xi)\in S_{d(v)}(X)$. This proves that if   $[F]\in\gM_v(X,H)^{\rm st}$ then  $c_2(F)\in S_{d(v)}(X)$: thus~\Ref{prp}{bastacont} gives that~\eqref{anvedi} holds. Now  suppose that $r=1$. Then $\gM_v(X,H)^{\rm st}$ ($=\gM_v(X,H)$) parametrizes sheaves $\cI_Z\otimes L$ where $Z\subset X$ is a $0$-dimensional subscheme of length $d(v)$: since $c_2(\cI_Z\otimes L)=[Z]$  we get that~\eqref{anvedi} holds by definition of $S_{d(v)}(X)$. 
\end{proof}
\subsection{Moduli spaces of sheaves}\label{subsec:invarianza}
\setcounter{equation}{0}
 Let  $X$ be a projective variety and $H$  an ample divisor on $X$. A torsion-free sheaf $F$ on $X$ is $H$-semistable if for every non-zero subsheaf $E\subset F$ we have
\begin{equation}\label{eulminore}
\frac{\chi(E\otimes\cO_X(xH))}{\rk E}\le \frac{\chi(F\otimes\cO_X(xH))}{\rk F}\qquad x\gg 0.
\end{equation}
 $F$  is $H$-stable if strict inequality holds for all $E\not=F$, it is  properly $H$-semistable if  it is $H$-semistable but not $H$-stable. 
 Let $E$ be a non-zero torsion-free sheaf on $X$: the $H$-slope of $E$ is defined to be
 \begin{equation}
\mu_H(E):=\frac{\deg (c_1(E)\cdot H)}{\rk E}.
\end{equation}
 A torsion-free sheaf $F$ is $H$-$\mu$-semistable if 
 if for every non-zero subsheaf $E\subset F$ we have $\mu_H(E)\le \mu_H(F)$; it is $H$-$\mu$-stable if strict inequality holds whenever $\rk(E)<\rk(F)$ and it is  properly $H$-$\mu$-semistable if  it is $H$-$\mu$-semistable but not $H$-$\mu$-stable. The two notions of (semi)stability are related as follows:  if $F$  is $H$-semistable then it is $H$-$\mu$-semistable, if $F$ is $H$-$\mu$-stable then it  is $H$-stable. 
\begin{prp}\label{prp:tensinv}
Let $X$ be a   $K3$ surface and $H$ an ample divisor on $X$. Let $v\in H^{\bullet}(X;\ZZ)$ be a Mukai vector. Let $y\in\ZZ$. Then~\eqref{anvedi} holds for $v$ if and only if it holds with $v$ replaced by $v\cdot \ch(\cO_X(yH))$.
\end{prp}
\begin{proof}
Let $w:=v\cdot \ch(\cO_X(yH))$. We have an isomorphism
\begin{equation*}
\begin{matrix}
\gM_v(X,H) & \overset{\sim}{\lra} & \gM_w(X,H) \\
[F] & \mapsto & [F\otimes\cO_X(yH)]
\end{matrix}
\end{equation*}
mapping $\gM_v(X,H)^{\rm st}$ to $\gM_w(X,H)^{\rm st}$.
We have  
\begin{equation}\label{tenscidue}
c_2(F\otimes\cO_X(yH))=c_2(F)+(r-1)y c_1(\cO_X(H))\cdot c_1(F) +{{r}\choose{2}}y^2 c_1(\cO_X(H))\cdot c_1(\cO_X(H)).
\end{equation}
 Hence the proposition follows from~\eqref{interseco}.
\end{proof}
Next we will compare  moduli spaces parametrizing torsion-free sheaves on a $K3$ surface $X$ with a fixed Mukai vector $v$ and (semi)stable with respect  to different ample divisors. The question has been studied more in general for arbitrary surfaces, see Appendix 4.C of~\cite{huy-lehn} and the references therein - here we will limit ourselves to the case of torsion-free sheaves on $K3$ surfaces. We will assume that $1\le r$ and $-2\le v^2$. Let
\begin{equation}
|v|:=\frac{r^2}{4} v^2 +\frac{r^4}{2}.
\end{equation}
Notice that  $|v|\ge 0$. Let $\Amp(X)\subset NS(X)$ be the set of ample classes and 
$\Amp(X)_{\RR}\subset NS(X)_{\RR}$ be the ample cone (here
$NS(X)_{\RR}:=H^{1,1}_{\RR}(X)$). A {\it $v$-wall} of $X$ consists of the interesection $\Amp(X)_{\RR}\cap \alpha^{\bot}$ where $\alpha\in NS(X)$ is such that
\begin{equation}
-|v|\le \alpha\cdot\alpha<0.
\end{equation}
The set of $v$-walls is  locally finite in $\Amp(X)_{\RR}$ - see for example Lemma~4.C.2 of~\cite{huy-lehn}. An {\it open $v$-chamber} of $\Amp(X)_{\RR}$ is a connected component of the complement  of the union of all $v$-walls in $\Amp(X)_{\RR}$. An ample divisor $H$ on $X$ is $v$-{\it generic} if its class belongs to an open $v$-chamber. 
The following result underscores the importance of $v$-walls and $v$-chambers - for the proof see the Appendix of~\cite{ogweight}.
\begin{prp}\label{prp:camere}
Let $X$ be a $K3$ surface and $v$ be a Mukai vector for $X$ with $r\ge 1$. 
\begin{enumerate}
\item[(1)]
Let $H$ be a $v$-generic ample divisor. Suppose that $F$ is a torsion-free properly $H$-slope-semistable  sheaf on $X$ with $v(F)=v$. Let $E\subset F$ be an $H$-slope destabilizing sheaf. Then 
\begin{equation}
\frac{c_1(E)}{\rk(E)}= \frac{c_1(F)}{\rk(F)}.
\end{equation}
\item[(2)]
Let $H_1,H_2$ be $v$-generic ample divisors whose classes belong to the {\bf same}  open $v$-chamber. Let $F$ be a  torsion-free sheaf on $X$  with $v(F)=v$.  Let $E\subset F$ be a non-zero subsheaf. Then $\mu_{H_1}(E)<\mu_{H_1}(F)$ if and only if $\mu_{H_2}(E)<\mu_{H_2}(F)$.
\item[(3)]
Suppose that $v$ is {\bf primitive} and $H$ is a $v$-generic ample divisor. Then $\gM_v(X,H)^{\rm st}=\gM_v(X,H)$.
\end{enumerate}
\end{prp}
Before proving the next result we will write out the normalized Hilbert polynomial of a sheaf $E$ on a $K3$ surface $X$.
\begin{equation}\label{hilbnorm}
\frac{\chi(E\otimes\cO_X(xH))}{\rk E}=\frac{1}{2}\deg(H\cdot H) x^2+\mu_H(E) x
+\frac{\chi(E)}{\rk(E)}.
\end{equation}
\begin{crl}\label{crl:camere}
Let hypotheses be as in~\Ref{prp}{camere}. Let $H_1,H_2$ be  $v$-generic   ample divisors whose classes belong to the {\bf same}  open $v$-chamber. 
A torsion-free sheaf on $X$  with $v(F)=v$ 
is $H_1$-(semi)stable  if and only if it is $H_2$-(semi)stable. We have an isomorphism
\begin{equation}\label{pjanic}
\begin{matrix}
\gM_v(X,H_1) & \overset{\sim}{\lra} & \gM_v(X,H_2)  \\
[F] & \mapsto & [F] 
\end{matrix} 
\end{equation}
\end{crl}
\begin{proof}
Assume that $F$ is $H_1$-stable. Suppose that $F$ is not $H_2$-stable. Let $E\subset F$ be a destabilizing sheaf i.e.~$0\not=E\not=F$ is non-zero and 
\begin{equation}
\frac{\chi(E\otimes\cO_X(xH))}{\rk E}\ge \frac{\chi(F\otimes\cO_X(xH))}{\rk F}\qquad x\gg 0.
\end{equation}
If $\mu_{H_2}(E)>\mu_{H_2}(F)$ then $\mu_{H_1}(E)>\mu_{H_1}(F)$ by Item~(2) of~\Ref{prp}{camere}: it follows that $F$ is not $H_1$-semistable, that contradicts our hypothesis. Thus 
\begin{equation}\label{pazzini}
\mu_{H_2}(E)=\mu_{H_2}(F),\qquad \frac{\chi(E)}{\rk(E)}\ge  \frac{\chi(F)}{\rk(F)}.
\end{equation}
By Item~(1) of~\Ref{prp}{camere} we get that 
\begin{equation*}
\mu_{H_1}(E)=\frac{\deg(c_1(E)\cdot H_1)}{\rk(E)}=\frac{\deg(c_1(F)\cdot H_1)}{\rk(F)}=\mu_{H_1}(F).
\end{equation*}
Since $F$ is $H_1$-stable it follows  that (see~\eqref{hilbnorm})
\begin{equation}
 \frac{\chi(E)}{\rk(E)}<  \frac{\chi(F)}{\rk(F)}.
\end{equation}
That contradicts~\eqref{pazzini}. We have proved that if $F$ is $H_1$-stable then it is $H_2$-stable. An easy application of Item~(1)  of~\Ref{prp}{camere} gives that  if $F$ is properly $H_1$-semistable then it is properly $H_2$-semistable. From this one gets that we have Isomorphism~\eqref{pjanic}.
\end{proof}
We close the present subsection by recalling the following result. 
\begin{thm}[Kuleshov~\cite{kuleshov}, Mukai~\cite{muktata}, Yoshioka~\cite{yoshi}]\label{thm:modce}
Let $X$ be a $K3$ surface. Let $v=(r,\ell,s)$ be a  Mukai vector for $X$ such that $-2\le v^2$. Suppose that  $r>0$ and that $r$ is coprime to the divisibility of $\ell$. Let $H$ be a $v$-generic ample divisor on $X$. Then $\gM_v(X,H)^{\rm st}$ is not empty (and it is equal to $\gM_v(X,H)$).
\end{thm}
\section{The degeneracy locus map}\label{sec:magari}
\setcounter{equation}{0}
In the present section we will suppose that the following hold:
\begin{enumerate}
\item[(1)]
 $X$ is a $K3$ surface and $H$ is an ample divisor on $X$.
\item[(2)]
$L$ is an {\bf ample} line-bundle on $X$. We let $\ell:=c_1(L)$ and
\begin{equation}\label{nonsolegg}
 v:=(r,\ell,s)\qquad r>0,\quad s\ge 0.
\end{equation}
\item[(3)]
 $[F]\in\gM_v(X,H)^{\rm st}$. 
\end{enumerate}
  By Serre duality $H^2(F)\cong \Hom(F,\cO_{X})^{\vee}$. By ampleness of $L$ and stability of $F$ we get that  $ \Hom(F,\cO_{X})=0$: it follows that $h^2(F)=0$. Thus
\begin{equation}\label{fenech}
h^0(F)=\chi(F)+h^1(F)\ge \chi(F)=r+s\ge r.
\end{equation}
 (The last inequality follows from~\eqref{nonsolegg}.)    Let $U\in\Gr(r,H^0(F))$:  we have the map of sheaves
\begin{equation}\label{gianburrasca}
\varphi_F^U\colon  U\otimes \cO_{X}\to  F.
\end{equation}
Let $\Gr(r,H^0(F))_{*}\subset \Gr(r,H^0(F))$ be the (open) subset of $U$ such that 
 $\det\varphi_F^U$ is non-zero: thus $\Gr(r,H^0(F))_{*}$ is non-empty if and only if global sections of $F$ generate $F$ generically. Let
\begin{equation}\label{dallagrass}
\begin{array}{ccl}
\Gr(r,H^0(F))_{*} & \overset{\lambda_F}{\lra} & |L| \\
U & \mapsto & C_F^U:=V(\det\varphi_F^U).
\end{array}
\end{equation}
This is the {\it degeneracy locus map} of $F$. 
\begin{lmm}\label{lmm:tiroindi}
Keep notation as above. The pull-back by $\lambda_F$ of the hyperplane class on $|L|$ is linearly equivalent to the Pl\"ucker hyperplane class on $\Gr(r,H^0(F))_{*}$. 
\end{lmm}
\begin{proof}
The natural map $\bigwedge^r H^0(F)\to H^0(\det F)$ induces a rational map
\begin{equation}\label{prepens}
\Lambda_F\colon \PP(\bigwedge^r H^0(F))\dra |L|. 
\end{equation}
Embed $\Gr(r,H^0(F))$  in $\PP(\bigwedge^r H^0(F))$ via Pl\"ucker: then  $\lambda_F$ is the restriction of $\Lambda_F$ to $\Gr(r,H^0(F))_{*}$: the claim follows.  
\end{proof}
\begin{prp}\label{prp:tiroindi}
Keep notation as above and suppose that  $\Gr(r,H^0(F))_{*}=\Gr(r,H^0(F))$. Then $c_2(F)\in S_{d(v)}(X)$.
\end{prp}
\begin{proof}
By~\Ref{lmm}{tiroindi} and the hypothesis that $\Gr(r,H^0(F))_{*}=\Gr(r,H^0(F))$ we get that $\lambda_F$ is finite. By~\eqref{fenech} it follows that  
 \begin{equation}\label{farfalle}
\dim\im\lambda_F\ge  rs.
\end{equation}
 By Hirzebruch-Riemann-Roch and Kodaira vanishing we have 
\begin{equation*}
 rs=1+\frac{\ell\cdot\ell}{2}-d(v)=\dim  |L|-d(v).
\end{equation*}
By~\eqref{almenogizero}   we have $\dim\Sigma_{d(v)}(X,L)\ge d(v)$.
   Thus~\eqref{farfalle} gives that there exists
\begin{equation*}
C\in(\im\lambda_F\cap \Sigma_{d(v)}(X,L)).
\end{equation*}
We have an exact sequence
\begin{equation*}
0\lra U\otimes\cO_{X}\lra F\lra \iota_{*}\xi\lra 0
\end{equation*}
where $\xi$ is a sheaf supported on $C$. Thus
\begin{equation*}
c_2(F)=-c_2(\iota_{*}\xi).
\end{equation*}
Since $C\in \Sigma_{d(v)}(X,L)$ we get that $c_2(F)\in S_{d(v)}(X)$ by~\Ref{clm}{spingo} and~\Ref{prp}{anakin}. 
\end{proof}
\begin{prp}\label{prp:verdone}
Suppose that the following holds: if $C$ is an effective {\bf non-zero} divisor on $X$ 
\begin{equation}\label{humptydumpty}
\deg(C\cdot H)>\frac{r-1}{r}\deg(L\cdot H).
\end{equation}
Then  $c_2(F)\in S_{d(v)}(X)$.
\end{prp}
\begin{proof}
Let's prove that  
\begin{equation}\label{stellatutto}
\Gr(r,H^0(F))_{*}=\Gr(r,H^0(F)).
\end{equation}
Let $U\in \Gr(r,H^0(F))$. Let $E\subset F$ be the image of the map 
$\varphi^U_F\colon U\otimes \cO_X\to F$: we must prove that it is a sheaf of rank $r$. Suppose that the rank of $E$ is $\ov{r}<r$. Since $E$ is globally generated $\det E$ is effective. We claim that $\det E$ is not trivial. In fact supppose the contrary. Let $\sigma_1,\ldots,\sigma_{\ov{r}}\in U$ be linearly independent at the generic point of $X$: then $\sigma_1,\ldots,\sigma_{\ov{r}}$ are  linearly independent everywhere because $\det E$ is trivial. Thus $E\cong\cO_X^{\ov{r}}$: that is absurd because $h^0(E)\ge r>\ov{r}=h^0(\cO_X^{\ov{r}})$. Let  $C\in |\det E|$. 
The sheaf $F$ is $H$-stable and hence $H$-slope-semistable. Thus
\begin{equation*}
\mu_H(E)=\frac{\deg(C \cdot H)}{\ov{r}}\le \mu_H(F)=\frac{\deg(L\cdot H)}{r}.
\end{equation*}
Since $C$ is non-zero effective and $\ov{r}<r$ that contradicts our hypothesis. We have proved that~\eqref{stellatutto} holds.  
By~\Ref{prp}{tiroindi}   it follows that   $c_2(F)\in S_{d(v)}(X)$. 
\end{proof}
\section{Primitive determinant}\label{sec:supposta}
\setcounter{equation}{0}
\begin{prp}\label{prp:picari}
Let $X$ be a $K3$ surface and $H$ an ample {\bf primitive} divisor on $X$ i.e.~$h:=c_1^{\rm hom}(\cO_X(H))$ is a primitive class. Let
\begin{equation*}
v:=(r,h,s),\qquad r>0,\quad s\ge 0.
\end{equation*}
Suppose that $\gM_v^{\rm st}(X,H)$ is not empty. Then
\begin{equation}
\{c_2(F) \mid [F]\in\ov{\gM}_v(X,H)^{\rm st} \}=\{\gz\in S_{d(v)}(X) \mid \deg\gz=c_2(v)\}.
\end{equation}
\end{prp}
\begin{proof}
 Let $g:=g(H)$ i.e.~$2g-2=\deg(H\cdot H)$. We will freely use notation introduced in~\Ref{subsec}{curvenodate}, in particular  $\pi\colon\cX\to T_g$ is a complete family of $K3$ surfaces with a polarization of degree $2g-2$. Thus there exists $\ov{t}\in T_g$ such that $(X_{\ov{t}},M_{\ov{t}})\cong(X,\cO_X(H))$.  
Let $t\in T_g$: we let 
\begin{equation*}
h_t:=c_1^{\rm hom}(M_t),\qquad v_t:=(r,h_t,s).
\end{equation*}
 Let $H_t\in |M_t|$. We will prove that if 
 $\gM_{v_t}^{\rm st}(X_t,H_t)$ is not empty then
\begin{equation}\label{borini}
\{c_2(F) \mid [F]\in\ov{\gM}_{v_t}(X_t,H_t)^{\rm st} \}=\{\gz\in S_{d(v_t)}(X_t) \mid \deg\gz=c_2(v_t)\}.
\end{equation}
Let $T_g(v)\subset T_g$ be the subset parametrizing $X_t$ such that $H_t$ is $v_t$-generic: then  $T_g(v)$ is  open dense in $T_g$. By~\Ref{thm}{modce} we get that if  $t\in T_g(v)$ then $\gM_{v_t}(X_t,H_t)^{\rm st}$ is not empty.  Let  $T_g(v)^0\subset T_g(v)$ be the  subset parametrizing $X_t$ such that the hypothesis of~\Ref{prp}{verdone} holds for $X=X_t$, $L=M_t$ and $H=H_t$. 
We claim that $T_g(v)^0$  is  open dense in $T_g$. In fact suppose that  there exists a non-zero effective divisor $C$ on $X_t$ violating~\eqref{humptydumpty}. Since $h_t$ is primitive it follows that  $c_1^{\rm hom}(\cO_{X_t}(C))$ does not belong to the $\QQ$-span of $h_t$, in particular the Picard number of $X_t$ is at least $2$. On the other hand $\deg(C\cdot H_t)$ is bounded above because~\eqref{humptydumpty} is violated: that implies that $t$ belongs to a proper closed subset of $T_g(v)$. Let $t\in T_g(v)^0$.  By~\Ref{prp}{verdone} we get that $c_2(F)\in S_{d(v_t)}(X_t)$ for all $[F]\in\ov{\gM}_{v_t}(X_t,H_t)^{\rm st}$. By~\Ref{prp}{bastacont} we get that~\eqref{borini} holds. Now let $t_0\in T_g$ be an arbitrary point such that $\ov{\gM}_{v_t}(X_t,H_t)^{\rm st}$ is not empty. Let $[F_0]\in \gM_{v_t}(X_t,H_t)^{\rm st}$. Then $\Ext^2(F_0,F_0)^0=0$ and hence $F_0$ extends sideways over the family $T_g$. It follows that there exist an irreducible pointed curve $(S,s_0)$, a map $f\colon S\to T_g$ such that 
\begin{equation}\label{ipad}
f(s_0)=t_0,\qquad f(S\setminus\{s_0\})\subset T_g(v)^0
\end{equation}
 and a (coherent) sheaf $\cF$ on $\cX_S:=S\times_{T_g}\cX$ flat over $S$ such that
\begin{enumerate}
\item[(1)]
$\cF|_{X_{s_0}}$ (for  $s\in S$ we let $X_s$ be the fiber of $\cX_S$ over $s$) is isomorphic to $F_0$ - this makes sense because $X_{s_0}=X_{t_0}$.
\item[(2)]
If $s\in S$ then $\cF|_{X_{s}}$ is an $H_s$-stable (for  $s\in S$ we identify $X_s\cong X_{f(s)}$ and we let $H_s:=H_{f(s)}$) torsion-free sheaf and hence its isomorphism class belongs to $\gM_{v_s}(X_s,H_s)^{\rm st}$ (here $v_s:=v_{f(s)}$).
\end{enumerate}
Let $s\in(S\setminus\{s_0\})$. By~\eqref{ipad} and the result  proved above we have that $c_2(\cF|_{X_{s}})\in S_{d(v_s)}(X_s)$. 
By~\Ref{clm}{specializzo} it follows that $c_2(\cF|_{X_{s_0}})\in S_{d(v_{s_0})}(X_{s_0})$ as well. This proves that if $t=t_0$ then the left-hand side of~\eqref{borini} is contained in the right-hand side: by~\Ref{prp}{bastacont} we get that the two sides of~\eqref{borini} are equal.
\end{proof}
The following result will be handy when we will deal with $K3$ surfaces whose Picard number is larger than $1$. 
\begin{lmm}\label{lmm:cambiopol}
Let $X$ be a $K3$ surface with $\rho(X)\ge 2$.  Let 
\begin{equation*}
v=(r,\ell,s)\in H^{\bullet}(X;\ZZ)
\end{equation*}
be a Mukai vector with $r\ge 1$ and such that $r$ and the divisibility of $\ell$ are coprime. 
Let $\cC\subset\Amp(X)_{\RR}$ be an open $v$-camber. There exists an integral $h\in\cC$ with the property that there is an infinite set of $y\in\NN$ such that $(\ell+ryh)$ is primitive.
\end{lmm}
\begin{proof}
Write $\ell=m\ell_1$ where $m\in\NN_{+}$ and $\ell_1\in H^2(X;\ZZ)$ is primitive. Complete $\ell_1$ to a $\ZZ$-basis $\{\ell_1,\ldots,\ell_{22}\}$ of $H^2(X;\ZZ)$. Let $h\in\cC$ be integral, primitive and such that  $\la\ell,h\ra$ has rank $2$ - there exists such $h$ because   $\rho(X)\ge 2$. 
Write $h=\sum_{i=1}^{22}a_i\ell_i$ where $a_i\in\ZZ$. Since $\la\ell,h\ra$ has rank $2$ there exists $2\le i\le 22$ such that $a_i\not=0$: thus it makes sense to let
\begin{equation*}
\gamma:=\gcd(a_2,\ldots,a_{22}).
\end{equation*}
Since open $v$-chambers are cones we may assume (changing slightly the ray spanned by $h$) that 
\begin{equation}\label{uovadoro}
\gcd(m,\gamma)=1.
\end{equation}
Let
\begin{equation}\label{fagiolimagici}
y=\gamma\cdot y_0,\qquad \gcd(y_0,m)=1.
\end{equation}
We will prove that $(\ell+ryh)$ is primitive. The proof is by contradiction. We have $\ell+ryh=(m+ry a_1)\ell_1+ry a_2\ell_2+\ldots+ ry a_{22}\ell_{22}$.  Suppose that $p$ is a prime dividing   $\ell+ryh$. Since  $\gcd(m,r)=1$ we get that $p\centernot\mid r$.  By~\eqref{uovadoro} and~\eqref{fagiolimagici}  we also get  that $p\centernot \mid y$. Since $p\mid(\ell+ryh)$ it follows that 
\begin{equation}\label{dividetimpera}
p\mid\gamma. 
\end{equation}
We also have that $p\mid(m+rya_1)$. Since $h$ is primitive $\gcd(a_1,\gamma)=1$ and hence $p\centernot \mid a_1$. It follows that  $y\equiv -r^{-1} a_1^{-1} m\pmod{p}$. By~\eqref{uovadoro}   we have that $m\not\equiv 0\pmod{p}$ and hence $y\not\equiv 0\pmod{p}$.     On the other hand~\eqref{fagiolimagici} and~\eqref{dividetimpera} give that $y\equiv 0\pmod{p}$: that is a contradiction. We have proved that if~\eqref{fagiolimagici} holds then $(\ell+ryh)$ is primitive. That proves the lemma because there is an infinite set of $y\in\NN$ such that~\eqref{fagiolimagici} holds. 
\end{proof}
\begin{prp}\label{prp:dirunio}
Let $X$ be a $K3$ surface with $\rho(X)\ge 2$.  Let 
\begin{equation*}
v=(r,\ell,s)\in H^{\bullet}(X;\ZZ)
\end{equation*}
be a Mukai vector with $r\ge 1$ and such that $r$ and the divisibility of $\ell$ are coprime. 
Let $H$ be a $v$-generic ample divisor on $X$. 
Then
\begin{equation}\label{natavota}
\{c_2(F) \mid [F]\in\ov{\gM}_v(X,H)^{\rm st} \}=\{\gz\in S_{d(v)}(X) \mid \deg\gz=c_2(v)\}.
\end{equation}
\end{prp}
\begin{proof}
The moduli space $\gM_v(X,H)^{\rm st}$ is not empty by~\Ref{thm}{modce}.  
Let $\cC\subset\Amp(X)_{\RR}$ be the open $v$-chamber containing $h:=c_1(\cO_X(H))$. By~\Ref{crl}{camere} we may replace $h$ by any integral element of $\cC$: thus we may suppose that $h$ is as in~\Ref{lmm}{cambiopol}.  Let  $y\in\NN$  be as in~\Ref{lmm}{cambiopol}; we will suppose in addition that it is very large. Then  $(\ell+ryh)$ is primitive,  ample and it belongs to the open $v$-chamber $\cC$. Thus we may assume that $c_1(\cO_X(H))=(\ell+ryh)$.   Let 
\begin{equation*}
w:=v\cdot \ch(\cO_X(yH))=(r,\ell+ryh,s+y\ell\cdot h+\frac{r}{2}y^2 h\cdot h). 
\end{equation*}
 The last entry of $w$ is positive because $y$ is very large. Thus the hypotheses of~\Ref{prp}{picari} hold with $v$ replaced by $w$ and hence~\eqref{natavota} holds with $v$ replaced by $w$.
By~\Ref{prp}{tensinv} we get that~\eqref{natavota} holds. 
\end{proof}
\begin{rmk}
Let $F$ be a rigid vector bundle  (a.k.a.~spherical vector-bundle) on a $K3$ surface $X$. An arbitrary (small) deformation of $X$ will carry  a rigid vector-bundle which is a  deformation of $F$, moreover the  deformed bundle will be (generically) stable - that follows from Proposition~3.14 of~\cite{muktata}. Starting from this fact and arguing as in the proof of~\Ref{prp}{picari} and~\Ref{prp}{dirunio} one  may reprove 
 Proposition~3.2 of~\cite{huybrechts} - one must notice that if $v=(r,\ell,s)$ is a Mukai vector with $v^2=-2$ then $s\ge 0$ and $r$ is coprime to the divisibility of $\ell$.  
\end{rmk}
\section{Rank two}\label{sec:coppiette}
\setcounter{equation}{0}
In the present section we will prove  the following result.
\begin{prp}\label{prp:rangodue}
Let $X$ be a $K3$ surface. Let $v=(2,\ell,s)\in H^{\bullet}(X,\ZZ)$ be a Mukai vector (notice: $r=2$) and suppose that the following hold:
\begin{enumerate}
\item[(1)]
$-2\le v^2$.
\item[(2)]
There does not exist $\ell_0\in H^{1,1}_{\ZZ}(X)$such that  $v=(2,2\ell_0,\ell_0\cdot\ell_0)$.
\end{enumerate}
 Let  $H$ be a $v$-generic  ample divisor on $X$.  Then
\begin{equation}\label{iamme}
\{c_2(F) \mid [F]\in\ov{\gM}_v(X,H)^{\rm st} \}=\{\gz\in S_{d(v)}(X) \mid \deg\gz=c_2(v)\}.
\end{equation}
\end{prp}
We start by collecting together a few results taken from  the existing literature on moduli of sheaves on $K3$ surfaces.
\begin{prp}\label{prp:kodzero}
Let $X$ be a $K3$ surface. Let $v=(2,\ell,s)\in H^{\bullet}(X,\ZZ)$ be a Mukai vector such that
\begin{equation}
-2\le v^2.
\end{equation}
 Let  $H$ be a $v$-generic  ample divisor on $X$. Then the following hold:
 \begin{enumerate}
\item[(1)]
$\gM^{st}_v(X,H)$ is empty if and only if $v=(2,2\ell_0,\ell_0\cdot\ell_0)$ for some $\ell_0\in H^{1,1}_{\ZZ}(X)$.
\item[(2)]
Suppose that $\gM^{st}_v(X,H)$ is not empty. The generic sheaf parametrized by $\gM^{st}_v(X,H)$ (recall that $\gM^{st}_v(X,H)$ is irreducible by~\cite{kls}) is locally-free unless $v=(2,2\ell_0,\ell_0\cdot\ell_0-1)$ for some $\ell_0\in H^{1,1}_{\ZZ}(X)$.
\item[(3)]
Suppose that $\gM^{st}_v(X,H)$ is not empty. The Kodaira dimension of $\gM^{st}_v(X,H)$ is $0$. 
\end{enumerate}
\end{prp}
\begin{proof}
(1): If the divisibility of $\ell$ is odd then $\gM^{st}_v(X,H)$ is not empty by~\Ref{thm}{modce}. Now assume   that  the divisibility of $\ell$ is  even i.e.~$v=(2,2\ell_0,s_0)$. 
Let $L_0$ be \lq\lq the\rq\rq line-bundle such that  
 $\ell_0=c_1^{\rm hom}(L_0)$. Tensorizing sheaves parametrized by $\gM^{st}_v(X,H)$ with $\cO_X(L_0^{-1})$ we  reduce to the case $v=(2,0,s)$ (because $H$ is $v$-generic). By hypothesis $s\le 0$. One checks easily that $\gM^{st}_v(X,H)$ is empty if $s=0$ (suppose that  $[F]\in\gM^{st}_v(X,H)$ and apply Hirzebruch-Riemann-Roch to $F^{\vee\vee}$). It remains to prove that $\gM^{st}_{(2,0,s)}(X,H)$ is non-empty if $s\le -1$. Choose pairwise distinct $p_1,\ldots,p_c\in X$ and pairwise distinct $K_1,\ldots,K_c\in\PP^1$. Let $F$ be the torsion-free sheaf fitting into the exact sequence
 \begin{equation}\label{barista}
0\to F\to \CC^2\otimes\cO_X\overset{\phi}{\lra} \bigoplus_{i=1}^c \CC_{p_i} \to 0.
\end{equation}
  One easily shows that if $c\ge 3$ then $F$ is $H$-stable. Since $v(F)=(2,0,2-c)$ we have  proved that $\gM^{st}_{(2,0,s)}(X,H)$ is non-empty for $s\le -1$. 
 (2): Suppose that  the divisibility of $\ell$ is odd. 
  Let $[F]\in\gM^{st}_v(X,H)$. If $F$ is locally-free there is nothing to prove. Assume that $F$ is not locally-free. The locally-free sheaf $F^{\vee\vee}$ is $H$-slope-semistable because $F$ is. Since $H$ is $v$-generic and  the divisibility of $c_1^{\rm hom}(\det F)$ is odd we get that $F^{\vee\vee}$ is $H$-slope-stable, in particular it is simple. As is well-known it follows that the generic deformation of $F$ is locally-free\footnote{Do a parameter count or examine the local-to-global spectral sequence abutting to $\Ext^p(F,F)$.}. Now suppose  that  the divisibility of $\ell$ is even. 
Arguing as in the proof of Item~(1) we may reduce to the case $v=(2,0,s)$. We must prove that if $s\le -2$ the  generic sheaf parametrized by $\gM^{st}_{(2,0,s)}(X,H)$ is locally-free. 
The moduli space  $\gM^{st}_{(2,0,-2)}(X,H)$ was investigated in~\cite{ogfortuna}: if $[F]\in\gM^{st}_{(2,0,-2)}(X,H)$ is generic then $F$ is locally-free and slope-stable. This proves the result for $s=-2$. By considering  deformations of torsion-free sheaves $E$ such that $E^{\vee\vee}\cong F$ one gets the result for  $s<-2$ as well.  
(3):  Suppose that $v$ is primitive: then
$\gM_v^{\rm st}(X,H)=\gM_v(X,H)$ because $H$ is $v$-generic, see Item~(3) of~\Ref{prp}{camere}. Thus  $\gM_v(X,H)$ is smooth and it  carries a holomorphic symplectic form:  it follows that it has trivial canonical bundle. Next suppose that $v$ is not primitive: we may reduce to the case $v=(2,0,2s_0)$ where $s_0\le -1$ i.e.~the moduli spaces investigated in~\cite{ogfortuna}. If $s_0=-1$ then $\gM_v(X,H)$ has a Hyperk\"ahler desingularization, if $s_0<-1$ then  $\gM_v(X,H)$ has a desingularization $\wh{\gM}_v(X,H)$  carrying a holomorphic $2$-form $\wh{\omega}$ which is generically non-degenerate (see Equation~(6.1) of~\cite{ogfortuna}), moreover one gets that the highest non-vanishing power $\bigwedge^{\rm max}\wh{\omega}$  generates the canonical ring of  $\wh{\gM}_v(X,H)$.
\end{proof}
\begin{rmk}\label{rmk:tuttising}
Let $X$ be a $K3$ surface. Let $v$ be the Mukai vector on $X$  given by $v=(2,2\ell_0,\ell_0\cdot\ell_0-1)$, see Item~(2) of~\Ref{prp}{kodzero}.  Let  $H$ be a $v$-generic  ample divisor on $X$. Let $L_0$ be \lq\lq the\rq\rq\ line-bundle such that $c_1^{\rm hom}(L_0)=\ell_0$. The generic sheaf parametrized by $\gM_v(X,H)$ fits into the exact sequence one gets by tensorizing~\eqref{barista}  (for $c=3$) with $L_0$:
 \begin{equation}\label{barolo}
0\to F\to \CC^2\otimes L_0\overset{\phi}{\lra} \bigoplus_{i=1}^3 \CC_{p_i} \to 0.
\end{equation}
Up to isomorphism the sheaf $F$ is independent of the choice of $K_1,K_2,K_3$ (notation as in the proof of~\Ref{prp}{kodzero}): thus the above  construction gives a birational map $X^{[3]}\dra\gM_v(X,H)$.
\end{rmk}
Let $X$ be a $K3$ surface and $L$ be  an  ample line-bundle on $X$.  Let us consider extensions 
\begin{equation}\label{severino}
0\to\cO_X\overset{\alpha}{\lra} F\overset{\beta}{\lra} \cI_Z\otimes L\to 0,\qquad [Z]\in X^{[n]}.
\end{equation}
Let $\ell:=c^{\rm hom}_1(L)$. We have 
\begin{equation}\label{impisi}
v(F)=(2,\ell,s),\qquad s=2+\frac{1}{2} \deg(\ell\cdot \ell)-n.
\end{equation}
\begin{prp}\label{prp:contopar}
Keep notation as above and let $H$ be an  ample divisor on $X$.  Let $v=(2,\ell,s)$.  Suppose that the generic sheaf $F$ parametrized by $\gM_v^{\rm st}(X,H)$ fits into Exact Sequence~\eqref{severino}. Then
\begin{equation}\label{carpos}
2+s> 0.
\end{equation}
\end{prp} 
\begin{proof}
We recall that $\gM_v^{\rm st}(X,H)$ is irreducible by~\cite{kls}. By our hypothesis there exist a strictly positive integer $e$ and an irreducible locally closed $V\subset X^{[n]}$ such that the following hold:
\begin{enumerate}
\item[(1)]
$\dim\Ext^1(\cI_Z\otimes L,\cO_X)=e$ for all $[Z]\in V$.
\item[(2)]
The generic $F$ parametrized by $\gM_v^{\rm st}(X,H)$
 fits into Exact Sequence~\eqref{severino} for some $[Z]\in V$. 
\end{enumerate}
 By Item~(1) there exists a locally-trivial (in the Zariski topology) $\PP^{e-1}$-bundle ${\bf E}\to V$ with fiber $\PP\Ext^1(\cI_Z \otimes L,\cO_X)$ over $[Z]\in V$. By Item~(2) the subset ${\bf E}^0\subset {\bf E}$ parametrizing stable sheaves is (open) dense and  the classification morphism $f\colon {\bf E}^0\to\gM_v^{\rm st}(X,H)$ is dominant.    By Item~(3) of~\Ref{prp}{kodzero} the moduli space  $\gM_v^{\rm st}(X,H)$ has Kodaira dimension equal to $0$, in particular it is not uniruled. It follows that  $f$ is constant on the fibers of  ${\bf E}^0\to V$ and hence
 \begin{equation*}
2n\ge\dim V\ge \dim\gM_v^{\rm st}(X,H)=2+v^2=4n-\deg (L\cdot L)-6.
\end{equation*}
Inequality~\eqref{carpos} follows at once from the above inequality together with~\eqref{impisi}.   
\end{proof}
The following hypothesis will be handy in what follows.
\begin{hyp}\label{hyp:comodo}
$\QQ c_1(L)=\QQ c_1(\cO_X(H))$ and the following holds: if $C$ is an effective divisor on $X$ such that 
\begin{equation}\label{walter}
\deg(C\cdot H)\le \frac{1}{2}\deg(L\cdot H)
\end{equation}
then $c_1(\cO_X(C))\in\QQ c_1(L)$. 
\end{hyp}
We let $m$ be the divisibility of $L$ i.e.
\begin{equation}\label{sottomultiplo}
L\cong L_0^m,\qquad \text{$\ell_0:=c^{\rm hom}_1(L_0)$ primitive.} 
\end{equation}
\begin{crl}\label{crl:contopar}
Keep notation as above, in particular  $v=(2,\ell,s)$. Suppose  that~\Ref{hyp}{comodo} holds and that
\begin{equation}\label{menodi}
2+s\le 0.
\end{equation}
 Then $h^0(F)=0$ for the generic sheaf $F$ parametrized by
 $\gM^{\rm st}_v(X,H)$.
\end{crl} 
\begin{proof}
We notice that Inequality~\eqref{menodi} implies that the generic sheaf parametrized by $\gM^{\rm st}_v(X,H)$ is locally-free: this follows from Item~(2) of~\Ref{prp}{kodzero} together with~\Ref{rmk}{tuttising}.  Suppose that $h^0(F)>0$ for the generic  $[F]\in\gM^{\rm st}_v(X,H)$.  Since the generic sheaf parametrized by $\gM^{\rm st}_v(X,H)$ is locally-free there exists an effective divisor $C$ such that we have an injection $\cO_X(C)\hra F$  for generic  $[F]\in\gM^{\rm st}_v(X,H)$. By $H$-slope-semistability of $F$ we get that~\eqref{walter} holds. By our hypothesis it follows that $\cO_X(C)\cong L_0^a$ for a non-negative integer (notation as in~\eqref{sottomultiplo}): by stability of $F$ we get that $a<m/2$ (a priori $a=m/2$ is also possible: a short argument shows that it is impossible). Thus the generic sheaf $F$ parametrized by $\gM^{\rm st}_v(X,H)$ fits into an exact sequence
\begin{equation}\label{aureliosaffi}
0\to L_0^a\to F\to\cI_Z\otimes L_0^{(m-a)}\to 0,\qquad 0\le a<m/2
\end{equation}
where $Z\subset X$ is a $0$-dimensional subscheme. 
 Let
\begin{equation}\label{gofmam}
v_a:=v(F\otimes L_0^{-a})=(2,(m-2a)\ell_0,s_a),\qquad s_a:=s-a(m-a) \ell_0\cdot \ell_0.
\end{equation}
Tensorization by $L_0^{-a}$ defines an isomorphism
\begin{equation*}
\begin{matrix}
\gM^{\rm st}_v(X,H) & \overset{\sim}{\lra} & \gM^{\rm st}_{v_a}(X,H) \\
[F] & \mapsto & [F\otimes L_0^{-a}]
\end{matrix}
\end{equation*}
because $\QQ c_1(L)=\QQ c_1(\cO_X(H))$. 
It follows that the generic $[G]\in \gM^{\rm st}_{v_a}(X,H)$ fits into an exact sequence
\begin{equation}
0\to\cO_X\to G\to\cI_Z\otimes L_0^{(m-2a)}\to 0.
\end{equation}
That contradicts~\Ref{prp}{contopar} because
\begin{equation*}
2+s_a=2+s-a(m-a) \ell_0\cdot \ell_0\le 2+s\le 0.
\end{equation*}
\end{proof}
The proof of the following result consists in realizing geometrically  a standard Fourier-Mukai transform. It  gives the proof that~\eqref{anvedi} holds for a  particular choice of Mukai vector $v$: we need to check that special case separately.  
\begin{prp}\label{prp:essemeno}
Keeping notation as above suppose  that  $L$ is ample and that~\Ref{hyp}{comodo} holds. Let  $v=(2,\ell,-1)$.   Let $[F]\in \gM^{\rm st}_{v}(X,H)$ be generic. Then  $F$ fits into Exact Sequence~\eqref{severino} where $n=(\deg(c_1(L)\cdot c_1(L))/2)+3$.  
\end{prp}
\begin{proof}
Let $n=(\deg(c_1(L)\cdot c_1(L))/2)+3$ and let $[Z]\in X^{[n]}$ be generic. Then  $Z$ satisfies the Cayley-Bacharach property with respect to the linear system $|L|$ and  up to isomorphism there exists a unique non-trivial extension~\eqref{severino}. Moreover 
the extension  $F$ is locally-free and $h^0(F)=1$. Now notice that 
$v(F)=(2,\ell,-1)=v$. We claim that $F$ is $H$-slope-stable. In fact suppose that $\iota\colon A\hra F$ is a destabilizing subsheaf i.e.~an invertible sheaf such that $2\deg(A\cdot H)\ge \deg(L\cdot H)$. Since $L$ is ample we get that $\deg(A\cdot H)>0$ and hence   $\iota$ cannot factor through $\alpha$ - notation as in~\eqref{severino}. Thus $\beta\circ\iota\not=0$ and hence $h^0(\cI_Z\otimes L\otimes A^{-1})>0$. Let $C\in | \cI_Z \otimes L\otimes A^{-1} |$. The inequality  $2\deg(A\cdot H)\ge \deg(L\cdot H)$ gives that~\eqref{walter} holds. It follows that $C\in | L_0^k |$ where $k\le m/2$ (notation as in~\eqref{sottomultiplo}) and hence  
$h^0(\cI_Z\otimes L_0^k)>0$: that is absurd because $h^0( L_0^k)<h^0(L)=n-1$ (recall that $Z$ is generic in $X^{[n]}$). 
 The above construction defines a rational map 
\begin{equation}\label{oldstuff}
X^{[n]}\dra \gM^{\rm st}_{v}(X,H)
\end{equation}
which is generically injective  (recall that with the above hypotheses $h^0(F)=1$ for the generic $F$ fitting into 
Extension~\eqref{severino}). The moduli space  $\gM^{\rm st}_{v}(X,H)$  is irreducible by~\cite{kls}; since 
\begin{equation}\label{numgiusto}
\dim X^{[n]}=2n=2+v^2=\dim\gM^{\rm st}_{v}(X,H)
\end{equation}
 it follows that~\eqref{oldstuff} is a birational map. This proves the proposition.
\end{proof}
\begin{crl}\label{crl:essemeno}
Keep hypotheses and notation as in~\Ref{prp}{essemeno}. Then
\begin{equation}\label{ciauscolo}
\{c_2(F) \mid [F]\in\ov{\gM}_v(X,H)^{\rm st} \}=\{\gz\in S_{d(v)}(X) \mid \deg\gz=c_2(v)\}.
\end{equation}
\end{crl}
\begin{proof}
Equation~\eqref{numgiusto} gives that $n=d(v)$: thus the corollary follows from~\Ref{prp}{essemeno}.
\end{proof}
\begin{prp}\label{prp:casochiave}
Keeping notation as above suppose  that  $L$ is ample and that~\Ref{hyp}{comodo} holds.
 Let $v=(2,\ell,s)$ be a Mukai vector with $s\ge 0$.   Then
\begin{equation}\label{finocchiona}
\{c_2(F) \mid [F]\in\ov{\gM}_v(X,H)^{\rm st} \}=\{\gz\in S_{d(v)}(X) \mid \deg\gz=c_2(v)\}.
\end{equation}
\end{prp}
\begin{proof}
Let $m$ be as in~\eqref{sottomultiplo}: then  $m\ge 1$ because $L$ is ample. We proceed by induction on $m$. If $m=1$  the statement holds by~\Ref{prp}{picari}. Now suppose that $m>1$. By~\Ref{clm}{specializzo} and~\Ref{prp}{bastacont} it suffices to prove that $c_2(F)\in S_{d(v)}(X)$ for the generic $[F]\in \gM^{\rm st}_v(X,H)$. \Ref{rmk}{tuttising} shows that~\eqref{finocchiona} holds if $v=(2,2\ell_0,\ell_0\cdot\ell_0-1)$: thus by Item~(2) of~\Ref{prp}{kodzero} we may assume that the generic sheaf parametrized by $\gM^{\rm st}_v(X,H)$ is locally-free. First assume that for the generic $[F]\in \gM^{\rm st}_v(X,H)$ we have $\Gr(2,H^0(F))_{*}=\Gr(2,H^0(F))$: then  $c_2(F)\in S_{d(v)}(X)$ by~\Ref{prp}{tiroindi}. Next we  assume that for the generic $[F]\in \gM^{\rm st}_v(X,H)$ we have $\Gr(2,H^0(F))_{*}\not=\Gr(2,H^0(F))$.  Let  $U\in(\Gr(2,H^0(F))\setminus\Gr(2,H^0(F))_{*})$ i.e.~the image of the map $U\otimes_{\CC}\cO_X\to F$ is a sheaf $\xi$ of rank $1$.  The sheaf $\xi$ is  invertible because $F$ is locally-free ($[F]$ is generic), since it has  $2$ independent sections there exists a non-zero effective divisor $C$ such that $\xi\cong\cO_X(C)$. By stability of $F$ we get that~\eqref{walter} holds.  By~\Ref{hyp}{comodo} it follows that 
$\xi\cong L_0^a$ for a strictly positive integer. Thus $F$ fits into Exact Sequence~\eqref{aureliosaffi} and $a>0$. Let $G:=F(-aH)$ and let $v_a$ be as in~\eqref{gofmam}. Then $[G]\in\gM_{v_a}(X,H)^{\rm st}$ is generic and $h^0(G)>0$. By~\Ref{crl}{contopar} it follows that $s_a\ge -1$. If $s_a=-1$ then~\eqref{finocchiona} holds for $v=v_a$ by~\Ref{crl}{essemeno}, if $s_a\ge 0$ then~\eqref{finocchiona} holds for $v=v_a$ by the inductive hypothesis (notice that $(m-2a)<m$). By~\Ref{prp}{tensinv} we get that~\eqref{finocchiona} holds for $\gM_{v}(X,H)^{\rm st}$.
\end{proof}

\vskip 2mm
\n
{\it Proof of~\Ref{prp}{rangodue}.\/} 
Tensorizing sheaves parametrized by  $\gM^{\rm st}_v(X,H)$ with high enough powers of $\cO_X(H)$ and replacing $h$ by $(\ell+2 n h)$ (for $n\gg 0$ it lies in the same $v$-chamber as $h$ does)  we may assume that $v=(2,mh,s)$ here
\begin{equation}
m>0,\qquad s\ge 0.
\end{equation}
Let $g:=g(H)$ i.e.~$2g-2=\deg(H\cdot H)$. Then there exists $\ov{t}\in T_g$ such that $(X,H)\cong (X_{\ov{t}},H_{\ov{t}})$. Given $t\in T_g$ we let $H_t\in |M_t|$ and $v_t:=(2,m h_t,s)$. 
We must prove that~\eqref{iamme} holds with $X=X_t$, $H=H_t$ and $v=v_t$ for an arbitrary $t\in T_g$. First we notice that $\gM_{v_t}(X_t,H_t)$ is not empty by~\Ref{prp}{kodzero}. Let $T_g(v)^{\star}\subset T_g$ be the set of $t$ such that~\Ref{hyp}{comodo} holds for $X=X_t$ and $L:=M_t^{m}$. Then  $T_g(v)^{\star}$ is an open dense subset of $T_g$. Let $t\in T_g(v)^{\star}$: by~\Ref{prp}{casochiave} Equality~\eqref{iamme} holds with $X=X_t$, $H=H_t$ and $v=v_t$.  By~\Ref{clm}{specializzo} - see the proof of~\Ref{prp}{picari} - it follows  that Equation~\eqref{iamme} holds with $X=X_t$, $H=H_t$ and $v=v_t$ and   arbitrary $t\in T_g(v)$. 
\qed
\section{Odds and ends}\label{sec:fantasia}
\setcounter{equation}{0}
\n
{\bf Simple versus stable sheaves.} Let  $F$ a simple sheaf on a  $K3$ surface $X$. Then 
\begin{equation*}
2-\dim\Ext^1(F,F)=-v(F)^2
\end{equation*}
and hence we may define $d(F)\in\NN$ by the formula
\begin{equation*}
2+v(F)^2=2d(F).
\end{equation*}
Below is a  natural question to ask:
\begin{qst}\label{qst:bastasempl}
Keep hypotheses and notation as above. Is it true that $c_2(F)\in S_{d(F)}(X)$ ?
\end{qst}
\n
{\bf The filtration and correspondences.} Let $X$ and $Y$ be $K3$ surfaces. Suppose that $\rho(X)\ge 2$ and $\rho(Y)\ge 2$. Let $\Phi_{\cE}\colon D^b(X)\overset{\sim}{\lra} D^b(Y)$ be a Fourier-Mukai  equivalence and let $\Phi_{\cE}^{A}\colon CH^{\bullet}(X)_{\QQ}\overset{\sim}{\lra} CH^{\bullet}(Y)_{\QQ}$ be the induced  isomorphism of additive groups (warning: the grading need not be respected). Huybrechts~\cite{huybrechts} proved that $\Phi_{\cE}^{A}$ maps the Beauville-Voisin ring of $X$ to  the Beauville-Voisin ring of $Y$. It is natural to ask  (no condition on $X$ and $Y$) the following question: which correspondence $\Gamma\in CH^2(X\times Y)$ respect the filtrations $S_{\bullet}(X)$ and  $S_{\bullet}(Y)$ ? 
Suppose that $f\colon X\dra Y$ is a rational map and  $\Gamma\in CH^2(X\times Y)$ is the graph of $f$: then $\Gamma$ maps $S_g(X)$ into  $S_g(Y)$.
\begin{expl}\label{expl:isoquo}
 Let $Y$ be a $K3$ surface with a symplectic  automorphism $f$. Let  $W:=Y/\la f\ra$ and $\pi\colon Y\to W$ be the quotient map. The  minimal desingularization of  $W$ is  a $K3$ surface $X$. Let $Z\subset X\times Y$ be the inverse image of 
$\{(\pi(p),p) \mid p\in Y\}$ via the map $X\times Y\to W\times Y$. 
Let $\Gamma:=[Z]\in CH^2(X\times Y)$. 
\end{expl}
Let $X$ and $Y$ be as in~\Ref{expl}{isoquo}.  Suppose that $f$ has order $2$. Then the couple $(X,Y)$ belongs to one of an infinite series of families which have been classified. The methods of Huybrechts and Kemeny~\cite{huy-kem} give that for many such families  $\Gamma_{*}(S_g(X))\subset S_g(Y)$ for all $g$. 
\vskip 2mm
\n
{\bf A filtration defined by Voisin.} Let $V$ be a smooth complex projective variety. In~\cite{voisin2} Voisin introduces the product 
\begin{equation*}
\begin{matrix}
CH_0(V)\times CH_0(V) & \lra & CH_0(V^2) \\
(\sum_i m_i p_i,\sum_j n_j q_j) & \mapsto & \sum_{i,j} m_i n_j (p_i,q_j)
\end{matrix}
\end{equation*}
One denotes the product of $Z_1$ and $Z_2$ by $Z_1 * Z_2$. Iterating we get $Z_1 * Z_2 *\ldots * Z_n\in CH_0(V^n)$. 
Voisin proved~\cite{voisin2}  that if $V$ is a curve of genus $g$ and $\deg Z=0$ then $Z^{* (g+1)}=0$. Now let $X$ be a $K3$ surface and let $S_g(X)_0\subset S_g(X)$ be the subset of degree-$0$ cycles. Let $Z\in S_g(X)_0$: then $Z$ is represented by a degree-$0$ cycle supported on a curve of geometric genus $g$ and hence $Z^{* (g+1)}=0$. Thus
\begin{equation}
S_g(X)_0 \subset \{Z\in CH_0(X) \mid \deg Z=0,\quad Z^{* (g+1)}=0\}. 
\end{equation}

\vskip 2mm
\n
{\bf Generalized Franchetta conjecture.} 
Let $g\ge 3$. Let $\gF_g$ be the moduli space of $K3$ surfaces with a polarization of degree $(2g-2)$. Let $\gF^0_g\subset \gF_g$ be the open dense subset parametrizing polarized $K3$ surfaces  with trivial automorphism group (of the polarized $K3$). There is a tautological family of $K3$ surfaces $\rho\colon\cX_g\to \gF^0_g$. The following question is  quite natural:
\begin{qst}\label{qst:franchettabis}
Let $\cZ\in CH^2(\cX_g)$. 
 Let $t\in \gF_g^0$ and  $X_t:=\rho^{-1}(t)$,  $Z_t=\cZ|_{X_t}$.
Is it true that $Z_t\in\ZZ c_{X_t}$ ?
\end{qst}
The statement of the above question is similar to Franchetta's conjecture on rationally defined line-bundles on the tautological family of curves on $\gM_g$ - now a Theorem, 
see~\cite{arbcor,mestrano}.  Franchetta's conjecture may be proved for very low values of $g$ by a simple direct argument. The proof may be adapted in order to  give an affirmative answer to~\Ref{qst}{franchettabis} for those values of $g$ such that the generic $K3$ surface of genus $g$ is a complete intersection in projective space i.e.~$g=3,4,5$ (I thank Daniel Huybrechts for bringing that to my attention). The  link between~\Ref{qst}{franchettabis}  and~\Ref{cnj}{solodim} is the following: if the answer to~\Ref{qst}{franchettabis} is affirmative then  $c_2(F)\in S_0(X)$ for any spherical vector-bundle $F$ on a $K3$ surface.

\end{document}